\documentclass[11pt,leqno]{article}      
\usepackage{graphicx}

\usepackage{latexsym}
\usepackage[cp1250]{inputenc}
\usepackage[OT4]{fontenc}
\usepackage{amsfonts}
\usepackage{amsmath}
\usepackage{amssymb}
\usepackage{url}

\newcommand{\R}{\mathbb{R}}

\newcommand{\Z}{\mathbb{Z}}

\newcommand{\Oo}{{\cal O}}

\newcommand{\rank}{\operatorname{rank}}
\newcommand{\codim}{\operatorname{codim}}
\newcommand{\inter}{\operatorname{int}}
\newcommand{\sgn}{\operatorname{sgn}}

\newcommand{\signature}{\operatorname{signature}}

\newcommand{\Aa}{\mathcal{A}}

\newcommand{\inv}{^{-1}}

\newtheorem{theorem}{Theorem}[section]
\newtheorem{lemma}[theorem]{Lemma}

\newtheorem{proposition}[theorem] {Proposition}

\newenvironment{example}{\medskip \noindent {\bf Example.\ }}{\bigskip}
\newenvironment{proof}{\par\noindent \emph{Proof. }}{\hspace*{\fill}$\Box$\par\medskip}

\title{Polynomial mappings into a Stiefel manifold and  immersions\thanks{%
Iwona~Krzy\.{z}anowska and Zbigniew~Szafraniec\\
University of Gda\'{n}sk,
              Institute of Mathematics \\
              80-952 Gda\'{n}sk, Wita Stwosza 57, Poland\\          
              Email: Iwona.Krzyzanowska@mat.ug.edu.pl\\
              Email: Zbigniew.Szafraniec@mat.ug.edu.pl\\ \\
2000 \emph{Mathematics Subject Classification}:  MSC 14P25, MSC 57R42\\
\emph{Keywords}: Stiefel manifolds, Immersions, Quadratic forms\\\\\emph{Supported by National Science Centre, grant 6093/B/H03/2011/40}
}
}

\author{Iwona Krzy\.{z}anowska \and Zbigniew~Szafraniec}

\date{November 2011}

\begin{document}

\def\nothanksmarks{\def\thanks##1{\protect\footnotetext[0]{\kern-\bibindent##1}}}

\nothanksmarks

\maketitle
%
%\pagestyle{fancy}
%
%\lhead{\fancyplain{}{\textsc{\small I.~Krzy\.{z}anowska, Z.~Szafraniec}}}
%\rhead{\fancyplain{}{\emph{\small Polynomial mappings into a Stiefel manifold and  immersions }}}
%

\begin{abstract} 
 For a polynomial mapping from
$S^{n-k}$ to the Stiefel manifold 
$\widetilde{V}_k(\R^{n})$, where $n-k$ is even,
there is presented an effective method of expressing the corresponding
element of the homotopy group $\pi_{n-k}\widetilde{V}_k(\R^{n})\simeq\Z$  
 in terms of signatures of quadratic forms.
There is also given a method of computing the intersection number
for a polynomial immersion $S^m\rightarrow\R^{2m}$.
%Insert your abstract here. Include keywords, PACS and mathematical
%subject classification numbers as needed.
%\keywords{Stiefel manifolds \and Immersions \and Quadratic forms}
% \subclass{MSC 14P25 \and MSC 57R42 }
\end{abstract}

\section{Introduction}
Mappings from a sphere into a Stiefel manifold are a natural object of study in topology.
Denote by $\widetilde{V}_k(\R^n)$ the non-compact Stiefel manifold, i.e. the set of all 
$k$--frames in $\R^n$, and take a polynomial mapping $\alpha:S^{n-k}\rightarrow \widetilde{V}_k(\R^n)$.
If $n-k$ is even then the homotopy group $\pi_{n-k}\widetilde{V}_k(\R^n)$ is isomorphic
to $\Z$. Let $\Lambda(\alpha)\in\Z$ be the integer associated with $\alpha$.

In Sections 2 we show that $\Lambda(\alpha)$ is equal to the topological degree
of some associated mapping
$\widetilde{\alpha}:S^{k-1}\times S^{n-k} \longrightarrow \R^n\setminus\{0\}$. 

In Section 3 we prove that one may express $\Lambda(\alpha)$ in terms of
signatures of two quadratic forms (Theorem \ref{efeektywnie}), even in the case where
$S^{n-k}$ is replaced by a compact algebraic hypersurface $M\subset \R^{n-k+1}$.
These signatures  may be computed using computer algebra systems.
Examples presented in this paper were calculated with the help of {\sc Singular}
\cite{GPS06}.

Assume that $m$ is even, $M\subset\R^{m+1}$ is a compact algebraic $m$--dimensional hypersurface,
and $g:M\rightarrow\R^{2m}$ is an immersion. Whitney in \cite{whitneySelfInter} introduced the intersection
number $I(g)\in\Z$. In the case where $M=S^m$, Smale in \cite{Smale} constructed a mapping
$\alpha':S^m\rightarrow \widetilde{V}_m(\R^{2m})$ such that $I(g)=\Lambda(\alpha')$.
Unfortunately, if $g$ is a polynomial immersion then $\alpha'$ is not a polynomial mapping,
so one cannot apply previous results in order to compute $I(g)$.

In Section 4  we show how to construct a polynomial mapping
$\alpha:M\rightarrow \widetilde{V}_{m+1}(\R^{2m+1})$ such that $I(g)=-\Lambda(\alpha)$ (Theorem \ref{immersje}).
Therefore $I(g)$ can be expressed and computed in terms of signatures.

Another formula expressing $I(g)$ in terms of signatures of quadratic forms,
inspired by the original definition by Whitney, was presented in \cite{KarNowSzafr}
and generalized in \cite{krzyzanowska} to the case where $M$ may have singularities.
Calculations done with the help of a computer show that the method presented in this paper
is significantly more effective.

\label{intro}
\section{Mappings into a Stiefel manifold}
\label{sec:1}

If $M,N$ are closed oriented $n$--manifolds and $f:M\longrightarrow N$ continuous, 
then by $\deg (f)$ we denote the topological degree of $f$. If $(M,\partial M)$ is
a compact oriented $n$--manifold with boundary and 
$f:(M,\partial M)\longrightarrow (\R^n,\R^{n}\setminus \{0\})$ is continuous,
then by $\deg (f|\partial M)$ we denote the topological degree of
$f/|f|:\partial M\longrightarrow S^{n-1}$.

Let $n-k> 0$ be an even number and $k>1$.
Denote by $\widetilde{V}_k(\R^n)$ the set of all 
$k$--frames in $\R^n$, and by $V_k(\R^n)$ the Stiefel manifold, 
i.e. the set of all orthonormal $k$--frames in $\R^n$. 
The Stiefel manifold is a deformation retract of $\widetilde{V}_k(\R^n)$, 
so that $\pi_{n-k}V_k(\R^n)=\pi_{n-k}\widetilde{V}_k(\R^n)$. 
It is known (see \cite{hatcher}), that $\pi_{n-k}V_k(\R^n)\simeq\Z$.

Let $[\alpha]\in \pi_{n-k}V_k(\R^n)$ be represented by 
$\alpha =(\alpha_1,\ldots ,\alpha_k):S^{n-k}\longrightarrow V_k(\R^n)$, 
where $\alpha_i:S^{n-k}\longrightarrow S^{n-1}\subset\R^n$. 
Since $\alpha_1(x),\ldots ,\alpha_{k}(x)$ are linearly independent, 
we can define $\widetilde{\alpha}:S^{k-1}\times S^{n-k}\longrightarrow \R^n\setminus \{0\}$ by 
$$\widetilde{\alpha}(\beta,x)=\beta_1\alpha_1(x)+\ldots+\beta_k\alpha_k(x),$$ 
where $\beta=(\beta_1\ldots ,\beta_k)\in S^{k-1}$ and $x=(x_1,\ldots ,x_{n-k+1})\in S^{n-k}$.

\begin{lemma}
The mapping
$\widetilde{\alpha}$ goes into $S^{n-1}$.
\end{lemma}

\begin{proof} We know that $|\beta|=1$, $| \alpha_i(x) |=1$ 
for $i=1,\ldots ,k$, and the scalar products 
$\langle\alpha_i(x),\alpha_j(x)\rangle=0$ for $i\neq j$. Then 
$$|\widetilde{\alpha}(\beta,x)|^2=
\sum_{i=1}^k\beta_i^2+2\sum_{i\neq j}\beta_i\beta_j\langle\alpha_i(x),\alpha_j(x)\rangle=1.$$
\end{proof}

We got $\widetilde{\alpha}:S^{k-1}\times S^{n-k}\longrightarrow S^{n-1}$, 
so the topological degree $\deg(\widetilde{\alpha})$ is defined.

For $\alpha:S^{n-k}\longrightarrow \widetilde{V}_k(\R^n)$, 
we also have $\widetilde{\alpha}:S^{k-1}\times S^{n-1}\longrightarrow \R^n\setminus \{0\}$. 
Then $\widetilde{\alpha}/|\widetilde{\alpha}|:S^{k-1}\times S^{n-1}\longrightarrow S^{n-1}$, 
and $\deg(\widetilde{\alpha}/|\widetilde{\alpha}|)$ is well defined. 
Let $r:\widetilde{V}_k(\R^n)\longrightarrow V_k(\R^n)$ be the retraction 
given by the Gram-Schmidt orthonormalization. Then  
$\alpha^t=(1-t)\alpha+t\cdot r\circ\alpha:S^{n-k}\longrightarrow\widetilde{V}_k(\R^n)$ 
is a homotopy between $\alpha$ and $r\circ\alpha$. Hence 
$\widetilde{\alpha}^t/|\widetilde{\alpha}^t|: S^{k-1}\times S^{n-k}\longrightarrow S^{n-1}$ 
is a homotopy between $\widetilde{\alpha}/|\widetilde{\alpha}|$ and $\widetilde{r\circ\alpha}$, 
so that $ \deg(\widetilde{\alpha}/|\widetilde{\alpha}|)=\deg(\widetilde{r\circ \alpha})$.

\begin{lemma}\label{homot}
If $\alpha^0,\alpha^1:S^{n-k}\longrightarrow V_k(\R^n)$ represent the same element in 
$\pi_{n-k}V_{k}(\R^n)$, then $\deg (\widetilde{\alpha}^0)=\deg(\widetilde{\alpha}^1)$.
\end{lemma}

\begin{proof} There is a homotopy 
$\alpha^t=(\alpha_1^t,\ldots , \alpha_k^t):S^{n-k}\longrightarrow V_k(\R^n)$ 
between $\alpha^0$ and $\alpha^1$. Then 
$$\widetilde{\alpha}^t(\beta,x)=\beta_1\alpha_1^t(x)+\ldots+\beta_k\alpha_1^t(x):
S^{k-1}\times S^{n-k}\longrightarrow S^{n-1}$$ 
is a homotopy between $\widetilde{\alpha}^0$ and $\widetilde{\alpha}^1$, 
and so $\deg (\widetilde{\alpha}^0)=\deg(\widetilde{\alpha}^1)$.

\end{proof}

\begin{example}\label{elneutralny}
The trivial element $[e]\in \pi_{n-k}V_k(\R^n)$ is represented by a constant mapping 
$e(x)=(v_1,\ldots ,v_k)\in V_k(\R^n)$. Then 
$\widetilde{e}(\beta,x)=\beta_1v_1+\ldots+\beta_kv_k$ is not onto $S^{n-1}$, so $\deg(\widetilde{e})=0$.
\end{example}

\begin{proposition}\label{parzystosc}
For all $[\alpha]\in \pi_{n-k}V_k(\R^n)$, $\deg(\widetilde{\alpha})$ is an even number.
\end{proposition}

\begin{proof} Let $M_k(\R^n)$ denote the set of all $k$--tuples of vectors in $\R^n$. 
According to \cite[Proposition 5.3]{golub}, 
$\Sigma_k(\R^n)=M_k(\R^n)\setminus \widetilde{V}_k(\R^n)$ is an algebraic subset of 
$M_k(\R^n)$ with $\codim \Sigma_k(\R^n)=n-k+1$. 

Take $\alpha :S^{n-k}\longrightarrow V_k(\R^n)$. There is a homotopy 
$h=(h_1,\ldots ,h_k):[0,1]\times S^{n-k}\longrightarrow M_k(\R^n)$ between 
$\alpha $ and $e$  given by $h(t,x)=te(x)+(1-t)\alpha(x)$. 
Put $h'=(h_2,\ldots ,h_k):[0,1]\times S^{n-k}\longrightarrow M_{k-1}(\R^n)$. 
According to the Elementary Transversality Theorem \cite[Corollary 4.12]{golub}, 
there is $g'$ arbitrarily close to $h'$ such that $g'$ intersects 
$\Sigma_{k-1}(\R^n)$ transversally. As $\codim\Sigma_{k-1}(\R^n)=n-k+2>\dim([0,1]\times S^{n-k})$, 
so $(g')\inv (\Sigma_{k-1}(\R^n))=\emptyset$ and then 
$g':[0,1]\times S^{n-k}\longrightarrow \widetilde{V}_{k-1}(\R^n)$. 
Put $g=(g_1,\ldots ,g_k)=(h_1,g')$.  We can choose $g'$ such that 
$g(0,\cdot):S^{n-k}\longrightarrow \widetilde{V}_{k}(\R^n)$ is homotopic to 
$\alpha$ and $g(1,\cdot):S^{n-k}\longrightarrow \widetilde{V}_{k}(\R^n)$ is homotopic to $e$. 
Put 
$\widetilde{g}:[0,1]\times S^{k-1}\times S^{n-k}\longrightarrow \R^n$ as  $\widetilde{g}(t,\beta,x)=
\beta_1g_1(t,x)+\ldots +\beta_kg_k(t,x)$.

 We know that $g_2(t,x),\ldots ,g_k(t,x)$ are linearly independent, 
so $\widetilde{g}(t,(0,\beta '),x)\neq 0$ for $\beta=(0,\beta ')\in \{0\}\times S^{k-2}$. 
Hence 
$$\widetilde{g}\inv(0)\cap ([0,1]\times (\{0\}\times S^{k-2})\times S^{n-k})=\emptyset  .$$
Put $H_{\pm}=\{\beta\in S^{k-1}|\ \pm\beta_1>0\}$,  and 
$M_{\pm}=[0,1]\times H_{\pm}\times S^{n-k}$. 
Then $\widetilde{g}\inv (0)\subset M_-\cup M_+$. Of course $i(t,\beta, x)=(t,-\beta ,x)$ 
is a free involution on $[0,1]\times S^{k-1}\times S^{n-k}$, such that 
$$\widetilde{g}\inv(0)\cap M_-=i(\widetilde{g}\inv(0)\cap M_+).$$  
We also have $\widetilde{g}\inv(0)\cap \{0,1\}\times S^{k-1}\times S^{n-k}=\emptyset$. 
So there exist two $n$--dimensional compact manifolds with boundary 
$N_{\pm}\subset \inter(M_{\pm})$  such that 
$\widetilde{g}(0)\subset \inter(N_-)\cup \inter(N_+)$ and $N_-=i(N_+)$. 
By the Excision Theorem, 

$$\deg(\widetilde{g}|\partial([0,1]\times S^{k-1}\times S^{n-k}))=
\deg(\widetilde{g}|\partial N_-)+\deg(\widetilde{g}|\partial N_+).$$ 
It is easy to check that $i:N_+\longrightarrow N_-$ preserves the orientation when $k$ is even, 
and  reverses it when $k$ is odd. 
We have  
$$\widetilde{g}|\partial N_+=-\widetilde{g}\circ i|\partial N_+=
-id|S^{n-1}\circ \widetilde{g}|\partial N_-\circ i|\partial N_+.$$ 
Of course $\deg (-id|S^{n-1})=(-1)^n$. Hence 
$$\deg(\widetilde{g}|\partial N_+)=(-1)^n(-1)^k\deg(\widetilde{g}|\partial N_-)=
(-1)^{n+k}\deg(\widetilde{g}|\partial N_-)=\deg(\widetilde{g}|\partial N_-).$$ 
On the other hand 
$$\deg(\widetilde{g}|\partial([0,1]\times S^{k-1}\times S^{n-k}))$$
$$=\deg(\widetilde{g}|(\{0\}\times S^{k-1}\times S^{n-k}))+\deg(\widetilde{g}|(\{1\}\times S^{k-1}\times S^{n-k}))$$
$$=\deg (\widetilde{e})-\deg (\widetilde{\alpha})=-\deg (\widetilde{\alpha}).$$ 
To sum up, we get that $\deg (\widetilde{\alpha})=-2\deg(\widetilde{g}|\partial N_-)$, so 
$\deg (\widetilde{\alpha})$ is even.

\end{proof}

Let us define 
$\Lambda:\pi_{n-k}V_k(\R^n)\longrightarrow \Z$ by $\Lambda([\alpha])=\deg(\widetilde{\alpha})/2$. 
According to  Proposition \ref{parzystosc}, $\Lambda$ is well defined.

\begin{proposition}
The mapping $\Lambda$ is a group homomorphism.
\end{proposition}

\begin{proof} Take $[\alpha_0],[\alpha_1]\in\pi_{n-k}V_k(\R^n)$, represented by smooth maps such that 
$\alpha_0(x_0)=\alpha_1(x_0)=y_0$. The sum $[\alpha_0]+[\alpha_1]$ in $\pi_{n-k}V_k(\R^n)$ 
is represented by the composition 
$(\alpha_0\vee \alpha_1)\circ c:S^{n-k}\longrightarrow S^{n-k}\vee S^{n-k}\longrightarrow V_k(\R^n)$, 
where $c$ collapses the equator $S^{n-k-1}$ in $S^{n-k}$ to a point, 
$x_0$ lies in $S^{n-k-1}$ and $S^{n-k}\vee S^{n-k}$ is the disjoint union of $S^{n-k}$ and 
$S^{n-k}$ with the identification $x_0\sim x_0$.
Then $\widetilde{\alpha_0+\alpha_1}:S^{k-1}\times S^{n-k}\longrightarrow S^{n-1}$. 
It is easy to see that $\widetilde{\alpha_0+\alpha_1}(S^{k-1}\times S^{n-k-1})$ 
is not dense in $S^{n-1}$. So there is a regular value $y\in S^{n-1}$ 
such that $\widetilde{\alpha_0+\alpha_1}\inv(y)\cap (S^{k-1}\times S^{n-k-1})=\emptyset$. 
Then $\widetilde{\alpha_0+\alpha_1}\inv(y)$ is the disjoint union of 
$\widetilde{\alpha}_0\inv(y)$ and $\widetilde{\alpha}_1\inv(y)$, so 
$\deg(\widetilde{\alpha_0+\alpha_1})=\deg(\widetilde{\alpha}_0)+\deg(\widetilde{\alpha}_1)$.

\end{proof}

\begin{proposition}
 The mapping $\Lambda$ is surjective.
\end{proposition}

\begin{proof} Let us define a mapping 
$\alpha=(\alpha_1,\ldots , \alpha_k) : S^{n-k}\longrightarrow V_k(\R^n)$ by 
$\alpha_i(x)=(0,\ldots , 1, 0, \ldots ,0)$, with the $1$ in the $i$--th coordinate for 
$i=1,\ldots ,k-1$, and $\alpha_k=(0,\ldots ,0,x_1,\ldots , x_{n-k+1})$. Then 
$$\widetilde{\alpha}(\beta, x)=
(\beta_1,\ldots , \beta_{k-1},\beta_kx_1,\ldots , \beta_{k}x_{n-k},\beta_kx_{n-k+1})\in S^{n-1}.$$ 
It is easy to check that 
$\widetilde{\alpha}\inv (0,\ldots ,0,1)=
\{(\beta,x)|\ \beta_1=\ldots=\beta_{k-1}=x_1=\ldots=x_{n-k}=0,\beta_k=x_{n-k+1}=
\pm 1\}$. 
Projection onto first $(n-1)$-- coordinates in a neighbourhood of $(0,\ldots ,0,1)\in S^{n-1}$ 
is an orientation preserving chart if
and only if $n$ is odd. Near $(0,\ldots ,0,1,0,\ldots ,0, 1)\in S^{k-1}\times S^{n-k}$ 
we have an orientation preserving parametrization given by
 $$(\beta',x')=(\beta_1,\ldots ,\beta_{k-1},x_1,\ldots ,x_{n-k})\longmapsto$$ 
$$\left((-1)^{k-1}\beta_1,\beta_2\ldots ,\beta_{k-1},\sqrt{1-|\beta'|^2},x_1,\ldots ,x_{n-k},\sqrt{1-|x'|^2}\right).$$ 
It is easy to check that in these coordinates the derivative matrix of 
$\widetilde{\alpha}$ at $(0,\ldots ,0,1,0,\ldots ,0, 1)$ has the form 
$$\left [
\begin{array}{cccc}
(-1)^{k-1}&0&\ldots &0\\
0&1&\ldots&0\\
&&\ddots&\\
0&0&\ldots &1
\end{array} \right ].$$
Near $(0,\ldots ,0,-1,0,\ldots ,0, -1)\in S^{k-1}\times S^{n-k}$ we have an orientation preserving parametrization 
 $$(\beta',x')=(\beta_1,\ldots ,\beta_{k-1},x_1,\ldots ,x_{n-k})\longmapsto $$
$$\left((-1)^k\beta_1,\beta_2,\ldots ,\beta_{k-1},-\sqrt{1-|\beta'|^2},-x_1,x_2,\ldots ,x_{n-k},-\sqrt{1-|x'|^2}\right).$$

It is easy to check that the derivative matrix of $\widetilde{\alpha}$ at 
$(0,\ldots ,0,-1,0,\ldots ,0, -1)$ in these coordinates has the form 
$$\left [
\begin{array}{ccccccc}
(-1)^k&0&\ldots &0&0&\ldots&0\\
0&1&\ldots &0&0&\ldots &0\\
&&\ddots&&&\ddots&\\
0&0&\ldots &1&0&\ldots &0\\
0&0&\ldots &0&-1&\ldots &0\\
&&\ddots&&&\ddots&\\
0&0& \ldots &0&0 &\ldots & -1\\
\end{array} \right ]$$
Then $(0,\ldots ,0,1)$ is a regular value of $\widetilde{\alpha}$, and 
$\deg(\widetilde{\alpha})=(-1)^{n-1}((-1)^{k-1}+(-1)^k(-1)^{n-k-1})=2$, so $\Lambda(\alpha)=1$.
By Proposition 2, $\Lambda$ is surjective.

\end{proof}

\begin{theorem}
If $n-k$ is even then $\Lambda:\pi_{n-k}V_k(\R^n)\longrightarrow\Z$ is  an isomorphism.
\end{theorem}

\begin{proof}
Since $\pi_{n-k}V_k(\R^n)\simeq\Z$, the surjective homomorphism $\Lambda$ is an isomorphism.
\end{proof}

Let $M$ be a closed oriented $(n-k)$--manifold. With any 
$\alpha:M\longrightarrow \widetilde{V}_{k}(\R^n)$ we  may associate the same way as above
the mapping
$\widetilde{\alpha}:S^{k-1}\times M\longrightarrow \R^n\setminus\{0\}$ given by 
$$\widetilde{\alpha}(\beta,x)=\beta_1\alpha_1(x)+\ldots+\beta_k\alpha_k(x).$$ 
Then the topological degree of $\widetilde{\alpha}$ is well defined.
Applying the same arguments as in the proof of Proposition \ref{parzystosc}, 
one can prove

\begin{theorem} \label{parzystoscstopnia}
Let $\alpha:M\longrightarrow \widetilde{V}_{k}(\R^n)$ be continuous, where $n-k$  is even.
Then $\deg( \widetilde{\alpha})$ is even, and so
 $\Lambda(\alpha)=\deg(\widetilde\alpha)/2$ is an integer. If
$\alpha^0,\alpha^1:M\rightarrow\widetilde{V}_k(\R^n)$ are homotopic, then
$\Lambda(\alpha^0)=\Lambda(\alpha^1)$.

\end{theorem}\hspace*{\fill}$\Box$\par\medskip
%%%%%%%%%%%%%%%%%%%%%%%%%%%%%%%%%%%%%%%%%%%%%%%%%%%%%%%%%%%%%%%%%%%%%%%%%%%%%%%%%%%%%%%%%%%%%%%%%%%%

\section{Polynomial mappings into a Stiefel manifold}\label{efektywnie}
If $U$ is an open subset of an $n$-dimensional oriented manifold, 
$H:U\longrightarrow \R^n$ is  continuous and $p\in H\inv(0)$ is isolated in 
$H\inv(0)$, then by $\deg_p H$ we denote the local topological degree of $H$ at $p$. 
If $H^{-1}(0)$ is compact, then there exists a compact manifold with boundary
$N\subset U$ such that $H^{-1}(0)\subset\operatorname{int}(N)$.
If that is the case then the topological degree $\deg(H,U,0)$ is defined as the degree
of the mapping $\partial N\ni x\mapsto H(x)/|H(x)|\in S^{n-1}$.
In particular, if $H^{-1}(0)$ is finite then $\deg(H,U,0)=\sum \deg_p H$, where 
$p\in H^{-1}(0)$.

Let $\alpha=(\alpha_1,\ldots ,\alpha_k):\R^{n-k+1}\longrightarrow M_k(\R^n)$ 
be a polynomial mapping. Denote by $\left [a_{ij}(x)\right ]$, 
$1\leq i\leq n$, $1\leq j\leq k$, the matrix in which $\alpha_j(x)$ 
stands in the $j$--th column. Define 
$$\widetilde{\alpha}(\beta, x)=
\beta_1\alpha_1(x)+\ldots +\beta_k\alpha_k(x)=\left [a_{ij}(x)\right ]\left [\begin{array}{c}
\beta_1\\
\vdots\\
\beta_k
\end{array} \right ]:\R^k\times \R^{n-k+1}\longrightarrow \R^n.$$
By $I$ we denote the ideal in $\R[x_1,\ldots ,x_{n-k+1}]$ generated by all 
$k\times k$ minors of  $\left [a_{ij}(x)\right ]$. Let $V(I)=\{x\in\R^{n-k+1}\ |\  h(x)=0\mbox{ for all }h\in I\}$.

\begin{lemma}\label{zaleznosc}
We have $p\in V(I)$ if and only if  $\alpha_1(p), \ldots ,\alpha_k(p)$ are linearly dependent, 
i.e. if  $\widetilde{\alpha}(\beta, p)=0$ for some $\beta\neq 0$.
\end{lemma}\hspace*{\fill}$\Box$\par\medskip

Let $I_1$ be the ideal generated by all $(k-1)\times(k-1)$ minors of $[a_{ij}(x)]$.
Put 
$$m(x)=\det\left [\begin{array}{ccc}
a_{12}(x)&\ldots &a_{1k}(x)\\ \\
a_{k-1,2}(x)&\ldots &a_{k-1,k}(x) 
\end{array} \right ].$$

\begin{lemma}\label{zaleznosc1}
We have $V(I)\setminus V(I_1)=\emptyset$ if and only if
$\rank [\alpha_1(p),\ldots,\alpha_k(p)]=k-1$ at each $p\in V(I)$.

If that is the case and $V(I)$ is finite then one may choose well oriented
 coordinates in $\R^{n-k+1}$ and $\R^n$
such that $m(p)\neq 0$ at each $p\in V(I)$.
\end{lemma}\hspace*{\fill}$\Box$\par\medskip

From now on we assume that  $m(p)\neq 0$ at each $p\in V(I)$. Hence, if 
$\alpha_1(p), \ldots ,\alpha_k(p)$ are linearly dependent
then $\alpha_2(p), \ldots ,\alpha_k(p)$ are linearly independent. 
 In that case there exists  a uniquely determined 
$\bar{\lambda}=(\bar{\lambda}_2,\ldots , \bar{\lambda}_k)\in \R^{k-1}$ 
such that $\alpha_1(p)+\bar{\lambda}_2\alpha_2(p)+\ldots +\bar{\lambda}_k\alpha_k(p)=0$.
Therefore we have

\begin{lemma}\label{jednoznacznosc}
Suppose that $p\in V(I)$. Then there is a uniquely  determined 
$\bar{\beta}\in H_+=S^{k-1}\cap \{\beta_1>0\}$ such that 
$\widetilde{\alpha}(\bar{\beta},p)=\bar{\beta}_1\alpha_1(p)+\bar{\beta}_2\alpha_2(p)+\ldots +\bar{\beta}_k\alpha_k(p)=0$,
and $\bar{\lambda}_2=\bar{\beta}_2/\bar{\beta}_1,\ldots,\bar{\lambda}_k=
\bar{\beta}_k/\bar{\beta}_1$.
\end{lemma}\hspace*{\fill}$\Box$\par\medskip

Note  that 
$$H_+\ni (\beta_1,\beta_2,\ldots, \beta_k)\mapsto (\beta_2/\beta_1,\ldots,\beta_k/\beta_1)\in \R^{k-1}$$ 
is an orientation  preserving diffeomorphism. For $\lambda=(\lambda_2,\ldots ,\lambda_k)\in \R^{k-1}$ 
and $1\leq i\leq n$
we define  $F_i(\lambda, x)=a_{i1}(x)+\lambda_2a_{i2}(x)+\ldots +\lambda_ka_{ik}(x)$. Then 
$$F=(F_1,\ldots , F_n)=\left [a_{ij}(x)\right ]\left [\begin{array}{c}
1\\ \lambda_2\\
\vdots\\
\lambda_k
\end{array} \right ]:\R^{k-1}\times \R^{n-k+1}\longrightarrow\R^n.$$

\begin{lemma}\label{FiLambda}
A point $(\bar{\beta}, p)\in H_+\times \R^{n-k+1}$ is an isolated zero of 
$\widetilde{\alpha}|(H_+\times\R^{n-k+1})$ if and only if $(\bar{\lambda},p)\in \R^{k-1}\times \R^{n-k+1}$
 is  an isolated zero of $F$. If that is the case then

 $$\deg_{(\bar{\beta}, p)}(\widetilde{\alpha}|(H_+\times\R^{n-k+1}))=\deg_{(\bar{\lambda},p)}(F).$$
\end{lemma}  \hspace*{\fill}$\Box$\par\medskip
In such case 
$$\frac{\partial (F_1,\ldots , F_{k-1})}{\partial (\lambda_2,\ldots , \lambda_k)}(\bar{\lambda},p)=m(p)\neq 0,$$
 so $\partial (F_1,\ldots , F_{k-1})/\partial (\lambda_2,\ldots , \lambda_k)\neq 0$
 in a neighbourhood of $(\bar{\lambda},p).$
 According to the Cramer rule, there exists a uniquely determined 
$\lambda(x)=(\lambda_2(x),\ldots ,\lambda_k(x))$
 defined in a neighbourhood of $p$, such that 

$$\left \{\begin{array}{c}
F_1(\lambda(x),x)=0\\
\vdots\\
F_{k-1}(\lambda(x),x)=0
\end{array}\right.$$ 
and \begin{equation}\lambda_i(x)=\frac{(-1)^{i-1}}{m(x)}\det\left[\begin{array}{ccccc}
a_{11}(x)&\ldots & \widehat{a_{1,i}(x)}&\ldots &a_{1k}(x)\\ &\ldots&&\ldots&\\
a_{k-1,1}(x)&\ldots &\widehat{a_{k-1,i}(x)}&\ldots &a_{k-1,k}(x)
\end{array} \right ]\  \label{star}\end{equation}
for $2\leq i\leq k$. 

Denote $\Gamma=\{(\lambda(x),x)\}\subset \R^{k-1}\times \R^{n-k+1}$, 
of course $\Gamma$ is an $(n-k+1)$-- manifold near $(\bar{\lambda},p)$. 
On $\Gamma$ we take the orientation induced by equations $F_1=\ldots =F_{k-1}=0$, 
i.e. vectors $v_1,\ldots, v_{n-k+1} $  tangent to 
$\Gamma$ at $(\lambda(x),x)$ are well oriented  if and only if 
$$\nabla F_1(\lambda(x),x),\ldots ,\nabla F_{k-1}(\lambda(x),x),v_1,\ldots , v_{n-k+1}$$ 
are well oriented in $\R^{k-1}\times\R^{n-k+1}$. One can see that vectors 
$$v_s=\frac{\partial}{\partial x_s}(\lambda(x),x)=
\Big(\frac{\partial \lambda_2}{\partial x_s},\ldots, 
\frac{\partial \lambda_k}{\partial x_s},0,\ldots, 1,0,\ldots 0\Big),\ 1\leq s\leq n-k+1,$$ 
with the $1$  in the $(k+s-1)$--th coordinate, are tangent to $\Gamma$. 
Their orientation corresponds to the standard orientation of an $\R^{n-k+1}$

\begin{lemma}\label{orientacja}
Vectors $v_1,\ldots, v_{n-k+1}$ are well oriented  if and only if $m(p)>0$.
\end{lemma}

\begin{proof} We have $F_i(\lambda(x),x)\equiv 0$, for $1\leq i\leq k-1$,  so
 $$0\equiv\frac{\partial}{\partial x_s}F_i(\lambda(x),x)$$
$$=a_{i2}(x)\frac{\partial \lambda_2}{\partial x_s}(x)+\ldots +a_{ik}(x)\frac{\partial \lambda_k}{\partial x_s}(x)+\frac{\partial a_{i1}}{\partial x_s}(x)+\lambda_2(x)\frac{\partial a_{i2}}{\partial x_s}(x)+\ldots +\lambda_k(x)\frac{\partial a_{ik}}{\partial x_s}(x).$$
 Denote by $A$ the Gramian matrix  of vectors $v_1,\ldots ,v_{n-k+1}$. It is easy to see that
the determinant of the matrix having rows
$\nabla F_1(\lambda(x),x),\ldots ,\nabla F_{k-1}(\lambda(x),x),$   $v_1(x),\ldots ,v_{n-k+1}(x)$ equals

$$\det\left[\begin{array}{cccccc}a_{12}(x)&\ldots &a_{1k}(x)&0&\ldots&0\\\vdots&&\vdots&\vdots&&\vdots\\a_{k-1,2}(x)&\ldots &a_{k-1,k}(x)&0&\ldots&0\\ &*&&&A
\end{array}\right ]=m(x)\det A.$$ 
Vectors $v_1,\ldots ,v_{n-k+1}$ are linearly independent, so $\det A>0$. 
We get that $v_1,\ldots ,v_{n-k+1}$ are well oriented in $\Gamma$ if and only if $ m(x)>0$.
\end{proof}

Denote by $\Oo_n$  the ring of germs at $(\bar{\lambda},p)$ of analytic functions
$\R^{k-1}\times\R^{n-k+1}\longrightarrow \R$, by  $\Oo_{\Gamma}$ the ring of germs at 
$(\bar{\lambda},p)$ of analytic functions
$\Gamma\longrightarrow \R$, and by $\Oo_{n-k+1}$  the ring of germs at $p$ of analytic functions
$\R^{n-k+1}\longrightarrow \R$. Since 
$\partial (F_1,\ldots , F_{k-1})/\partial (\lambda_2,\ldots , \lambda_k)(\bar{\lambda},p)\neq 0$, 
so there  is a natural isomorphism  
$$\Oo_{\Gamma}\simeq \Oo_n/\langle F_1,\ldots , F_{k-1}\rangle .$$ 
Put $\Omega_i(x)=F_i(\lambda(x),x)\in \Oo_{n-k+1}$. Because $\Gamma$ is the graph of 
$x\mapsto \lambda(x)$, so $\Oo_{\Gamma}\simeq \Oo_{n-k+1}$ and 
$$\Oo_n/\langle F_1,\ldots ,F_{k-1},F_k,\ldots , F_n\rangle \simeq \Oo_{\Gamma} / \langle F_k,
\ldots , F_n\rangle \simeq \Oo_{n-k+1}/\langle \Omega_1,\ldots ,\Omega_n \rangle.$$ 

We have $m(p)\neq 0$, so the germ of $m$ is invertible in $\Oo_n$, $\Oo_{\Gamma}$ and $\Oo_{n-k+1}$. 
Denote by $J$ the ideal in $\Oo_{n-k+1}$ generated by all $k\times k$ minors of $\left [a_{ij}(x)\right ]$.

\begin{lemma}
$\Oo_n/\langle F_1,\ldots , F_n\rangle \simeq \Oo_{\Gamma} / \langle F_k,\ldots , F_n\rangle\simeq \Oo_{n-k+1}/J$.
\end{lemma}

\begin{proof}
Take $1\leq i_1<\ldots <i_k\leq n$. Of course 
$\Omega_{i_1}=\ldots =\Omega_{i_k}=0$ in $\Oo_{n-k+1}/\langle \Omega_1,\ldots , \Omega_n\rangle$, 
i.e.
$$\left \{\begin{array}{c}
a_{i_1,1}(x)+\lambda_2(x)a_{i_1,2}(x)+\ldots +\lambda_k(x)a_{i_1,k}(x)=0\\
\vdots\\
a_{i_k,1}(x)+\lambda_2(x)a_{i_k,2}(x)+\ldots +\lambda_k(x)a_{i_k,k}(x)=0
\end{array}\right.$$ 
According to the Cramer rule, 
$$\det\left [\begin{array}{ccc}
a_{i_1,1}(x)&\ldots &a_{i_1,k}(x)\\ &\ldots& \\
a_{i_k,1}(x)&\ldots &a_{i_k,k}(x)
\end{array} \right ]=0\ \mbox{ in }\ \Oo_{n-k+1}/\langle \Omega_1,\ldots,\Omega_n\rangle .$$
 Each generator of $J$ belongs to $\langle \Omega_1,\ldots , \Omega_n\rangle$, 
so $J\subset \langle \Omega_1,\ldots , \Omega_n\rangle$.

Applying  (\ref{star}) one may show 
$$\Omega_i(x)=\frac{(-1)^{k-1}}{m(x)}\det\left[\begin{array}{ccc}
a_{11}(x)&\ldots &a_{1k}(x) \\
 &\ldots& \\
a_{k-1,1}(x)&\ldots &a_{k-1,k}(x)\\
a_{i1}(x)&\ldots &a_{ik}(x)
\end{array} \right ]=:\frac{(-1)^{k-1}}{m(x)}\Delta_i(x).$$
In particular,  germs of $\Omega_1,\ldots, \Omega_n$ in $\Oo_{n-k+1}$ belong to $J$. 
Of course $\Omega_1 \equiv 0,\ldots,\Omega_{k-1}\equiv 0$.
So $\langle \Omega_1,\ldots ,\Omega_n\rangle=\langle  \Omega_k,\ldots,\Omega_n \rangle=J$, and  there are natural isomorphisms   
$$\Oo_n/\langle F_1,\ldots , F_n\rangle \simeq \Oo_{\Gamma} / \langle F_k,\ldots , 
F_n\rangle\simeq \Oo_{n-k+1}/\langle \Omega_1,\ldots ,\Omega_n \rangle\simeq\Oo_{n-k+1}/J.$$ \end{proof}
From the previous proof we also get 

\begin{lemma}
$\Oo_{n-k+1}/J\simeq \Oo_{n-k+1}/\langle \Delta_k,\ldots,\Delta_n\rangle$.

\end{lemma}\hspace*{\fill}$\Box$\par\medskip

\begin{lemma}\label{stopienF}
If $n-k$ is even and $\partial (\Delta_k,\ldots , \Delta_n)/\partial (x_1,\ldots ,x_{n-k+1})(p)\neq 0$, 
then $$\deg_{(\bar{\lambda},p)}(F)=(-1)^{k-1}\sgn\frac{\partial (\Delta_k,\ldots ,
 \Delta_n)}{\partial (x_1,\ldots ,x_{n-k+1})}(p).$$

\end{lemma}

\begin{proof} We have $\Omega_i(x)m(x)=(-1)^{k-1}\Delta_i(x)$ so 

$$\frac{\partial\Omega_i}{\partial x_s}m+\Omega_i\frac{\partial m}{\partial x_s}=
(-1)^{k-1}\frac{\partial \Delta_i}{\partial x_s}.$$ 
As $\Omega_i\in J$, then 
$$\frac{\partial\Omega_i}{\partial x_s}\equiv \frac{(-1)^{k-1}}{m}\frac{\partial\Delta_i}{\partial x_s}\ \mod J.$$ 
Because $n-k$ is even, so 
$$\frac{\partial (\Omega_k,\ldots , \Omega_n)}{\partial (x_1,\ldots ,
x_{n-k+1})}\equiv \frac{(-1)^{k-1}}{m^{n-k+1}}\frac{\partial (\Delta_k,\ldots , \Delta_n)}
{\partial (x_1,\ldots ,x_{n-k+1})}\ \mod J.$$ 
Since $\partial (\Delta_k,\ldots , \Delta_n)/\partial (x_1,\ldots ,x_{n-k+1})(p)\neq 0$, 
the mapping $(\Omega_k,\ldots ,\Omega_n)$ has an isolated regular zero at $p$, moreover 
$$\deg_p(\Omega_k,\ldots,\Omega_n)=
(-1)^{k-1}\sgn (m(p))\sgn\frac{\partial (\Delta_k,\ldots , \Delta_n)}{\partial (x_1,\ldots ,x_{n-k+1})}(p).$$
So the local topological degree of 
$(F_k,\ldots ,F_n):(\Gamma,(\bar{\lambda},p))\longrightarrow (\R^{n-k+1},0)$ at  $(\bar{\lambda},p)$ equals 
$$(-1)^{k-1}\sgn (m(p))\sgn\frac{\partial (\Delta_k,\ldots , \Delta_n)}{\partial (x_1,\ldots ,x_{n-k+1})}(p),$$ 
if and only if the orientation of $\Gamma$ is the same as the one induced from $\R^{n-k+1}$. 
According to Lemma \ref{orientacja}, the local topological degree of 
$(F_k,\ldots ,F_n):(\Gamma,(\bar{\lambda},p))\longrightarrow (\R^{n-k+1},0)$ , 
where the orientation of $\Gamma$ is the one induced by equations $F_1=\ldots=F_{k-1}=0$, equals 
$$(-1)^{k-1}\sgn\frac{\partial (\Delta_k,\ldots , \Delta_n)}{\partial (x_1,\ldots ,x_{n-k+1})}(p).$$ 
According to \cite[Lemma 3.2]{Szafraniec}, $(\bar{\lambda}, p)$ is isolated in 
$((F_k,\ldots ,F_n)|\Gamma)\inv(0)$ if and only if 
$(\bar{\lambda}, p)$ is isolated in 
$(F_1,\ldots ,F_n)\inv(0)$, moreover 
$$\deg_{(\bar{\lambda}, p)}(F_k,\ldots ,F_n)|\Gamma=\deg_{(\bar{\lambda}, p)}(F_1,\ldots ,F_n).$$ 
Hence the local topological degree of  
$F=(F_1,\ldots, F_{k-1},F_{k},\ldots,F_n):(\R^{k-1}\times \R^{n-k+1},(\bar{\lambda}, p))\longrightarrow
(\R^n,0)$ at $(\bar{\lambda}, p)$ equals 
$$(-1)^{k-1}\sgn\frac{\partial (\Delta_k,\ldots , \Delta_n)}{\partial (x_1,\ldots ,x_{n-k+1})}(p).$$
\end{proof}

Put $\Aa=\R[x_1,\ldots,x_{n-k+1}]/I$. Let us assume that $\dim\Aa<\infty$, so that $V(I)$ is finite.
For $h\in\Aa$, we denote by $T(h)$ the trace of the linear endomorphism
$\Aa\ni a\mapsto h\cdot a\in\Aa$. Then $T:\Aa\rightarrow\R$ is a linear functional.

Let $f\in\R[x_1,\ldots,x_{n-k+1}]$ and $M=f^{-1}(0)$. Assume that $D=\{x\ |\ f(x)\geq 0\}$ is bounded and
$\nabla f(x)\neq 0$ at each $x\in M$. Then $D$ is a compact manifold with boundary
$\partial D=M$, and $\dim M=n-k$. 

Put $\delta=\partial(\Delta_k,\ldots,\Delta_n)/\partial(x_1,\ldots,x_{n-k+1})$. With $f$ and $\delta$ we associate
quadratic forms $\Theta_\delta,\ \Theta_{f\cdot \delta}:\Aa\rightarrow\R$
given by $\Theta_\delta(a)=T(\delta\cdot a^2)$ and $\Theta_{f\cdot \delta}(a)=T(f\cdot\delta\cdot a^2)$.
According to \cite{becker,pedersenetal}, we have
$$\operatorname{signature}\, \Theta_\delta=\sum \operatorname{sgn}(\delta(p)),$$
$$\operatorname{signature}\, \Theta_{f\cdot \delta}=\sum \operatorname{sgn}(f(p)\delta(p)),$$
where $p\in V(I)$. Moreover, if the quadratic forms are non-degenerate then
$\delta(p)\neq 0$ and $f(p)\neq 0$ at each $p\in V(I)$, so that $V(I)\cap M=\emptyset$.
In that case vectors $\alpha_1(x),\ldots,\alpha_k(x)$ are linearly independent
at every $x\in M$, and then the restricted mapping $\alpha|M$ goes into $\widetilde{V}_k(\R^n)$.
Hence $\widetilde{\alpha}|S^{k-1}\times M$ goes into $\R^n\setminus\{0\}$.

\begin{theorem}\label{efeektywnie} If $n-k$ is even, 
$\alpha=(\alpha_1,\ldots ,\alpha_k):\R^{n-k+1}\longrightarrow M_k(\R^n)$ 
is a polynomial mapping such that $\dim\Aa<\infty$, 
$I+\langle m \rangle=\R[x_1,\ldots ,x_{n-k+1}]$ and quadratic forms 
$\Theta_\delta,\, \Theta_{f\cdot\delta}:\Aa\longrightarrow\R$ are non--degenerate then 
$$\Lambda(\alpha|M)=\frac{1}{2}\deg(\widetilde{\alpha}|S^{k-1}\times M)=
\frac{1}{2}(\signature\Theta_{\delta}+\signature\Theta_{f\cdot\delta}),$$ 
where $\widetilde{\alpha}(\beta,x)=\beta_1\alpha_1(x)+\ldots +\beta_k\alpha_k(x)$.

\end{theorem} 

\begin{proof}
The product $S^{k-1}\times D$ is a compact $n$--manifold with boundary
$\partial(S^{k-1}\times D)=S^{k-1}\times M$. The standard orientation of the boundary
coincides with the standard orientation of the product $S^{k-1}\times M$ if and only if $k$ is odd.
Then
$$\deg(\widetilde{\alpha}|S^{k-1}\times M)=(-1)^{k-1}\deg(\widetilde{\alpha},S^{k-1}\times D,0).$$

Take $(\bar{\beta},p)\in\widetilde{\alpha}^{-1}(0)\cap S^{k-1}\times D$.
Then $\alpha_1(p),\ldots,\alpha_k(p)$ are linearly dependent.
By Lemma \ref{zaleznosc}, $p$ belongs to a finite set $V(I)$. Since $I+\langle m\rangle=\R[x_1,\ldots,x_{n-k+1}]$,
we have $m(p)\neq 0$ at each $p\in V(I)$. Then,
by Lemma \ref{jednoznacznosc}, $(\bar{\beta},p)$ and $(-\bar{\beta},p)$ are isolated in
$\widetilde{\alpha}^{-1}(0)\cap S^{k-1}\times D$ and the first coordinate
$\bar{\beta}_1\neq 0$.

Because $n-k$ is even, then
$$\deg_{(\bar{\beta},p)}\widetilde{\alpha}=  
\deg(-id|S^{n-1})\cdot\deg_{(-\bar{\beta},p)}\widetilde{\alpha}\cdot \deg(-id|S^{k-1})=
\deg_{(-\bar{\beta},p)}\widetilde{\alpha}.$$
Hence $\deg(\widetilde{\alpha},S^{k-1}\times D,0)=2\deg(\widetilde{\alpha},H_+\times D,0)$. 
By Lemma \ref{FiLambda}, 
$$\deg(\widetilde{\alpha},H_+\times D,0)=\deg(F,\R^{k-1}\times D,0)=\sum\deg_{(\bar{\lambda},p)} F,$$
where $(\bar{\lambda},p)\in F^{-1}(0)\cap \R^{k-1}\times D$.

The quadratic form $\Theta_\delta$ is non-degenerate, hence
$\delta(p)\neq 0$ at each $p\in V(I)$. By Lemma \ref{stopienF},
$$\deg(\widetilde{\alpha},S^{k-1}\times D,0)=2(-1)^{k-1}\sum\sgn \delta(p),$$
where $p\in V(I)\cap D=V(I)\cap\{f>0\}$.

On the other hand,
$$\signature\Theta_\delta+\signature\Theta_{f\cdot \delta}$$
$$=\sum_{p\in V(I)}\sgn\delta(p)+\left(   \sum_{p\in V(I)\cap\{f>0\}}\sgn\delta(p)-    \sum_{p\in V(I)\cap\{f<0\}}\sgn\delta(p)         \right)$$
$$=2\sum_{p\in V(I)\cap D}\sgn\delta(p).$$

\end{proof}

\begin{example}
Take $$A=\left[\begin{array}{cc}
2z+2&y+2\\2y+1&2y+1\\2x+1&y+2\\z+1&2y+1

\end{array}\right].$$ Let $\alpha=(\alpha_1,\alpha_2):\R^3\longrightarrow M_2(\R^4)$ be a polynomial mapping such that $\alpha_j$ is the $j$--th column of $A$. 

One may check that $I$ is generated by $2y-z,$ $2x-2z-1,$ $z^2+z$ and $\Aa=\R[x,y,z]/I$ is a $2$--dimensional algebra, where $e_1=1$, $e_2=z$ is its basis. 

In our case $m=y+2$, and so $I+\langle m\rangle=\R[x,y,z].$ One may check that $T(e_1)=2$, $T(e_2)=-1$, and $\delta=-24-\frac{75}{2}z$, $f\cdot \delta =-18-\frac{45}{4}z$ in $\Aa$. The matrices of $\Theta_{\delta}$ and $\Theta_{f\cdot \delta}$ are 
$$\left[\begin{array}{cc}
-{21}/{2}&-{27}/{2}\\
-{27}/{2}&{27}/{2}
\end{array}\right] \ \mbox{ and }\ \left[\begin{array}{cc}
-{99}/{4}&{27}/{4}\\
{27}/{4}&-{27}/{4}
\end{array}\right].$$
Hence $\signature \Theta_{\delta}=0$, $\signature \Theta_{f\cdot \delta}=-2$. Applying  Theorem \ref{efeektywnie} we get that $\alpha|S^2:S^2\longrightarrow \widetilde{V}_2(\R^4)$ and $\Lambda(\alpha|S^2)=(0-2)/2=-1.$
\end{example}

%%%%%%%%%%%%%%%%%%%%%%%%%%%%%%%%%%%%%%%%%%%%%%%%%%

\section{Intersection number}
\label{sec:2}

By $B^n(r)$ we denote the $n$--dimensional open ball with radius $r$ centred at the origin,
 by $\bar{B}^n(r)$ its closure,
 and by $S^{n-1}(r)$ the $(n-1)$--dimensional sphere with radius $r$ centred at the origin.

\begin{lemma}\label{stopienNaSferze}
Let $H=(h_1,\ldots ,h_n):\R^n\longrightarrow \R^n$ be continuous. If 
$Z=\{x\in \R^n|\ h_1(x)=\ldots =h_k(x)=0\}$ is compact for some $k<n$, 
then the topological degree of $H/|H|:S^{n-1}(R)\longrightarrow S^{n-1}$ 
is equal to zero for any $R>0$ with $Z\subset B^n(R).$
\end{lemma}

\begin{proof} Suppose that $Z\subset B^n(R)$, so $Z\cap S^{n-1}(R)=\emptyset$ and 
$H(x)\neq 0$ for $x\in S^{n-1}(R)$. Let us consider 
$H/|H|:S^{n-1}(R)\longrightarrow S^{n-1}$. We have 
$$ (H/|H|)\inv(0,\ldots ,0,1)\subset Z\cap S^{n-1}(R)=\emptyset.$$ 
So $(H/|H|)\inv(0,\ldots ,0,1)=\emptyset$, and the degree of $H/|H|$ equals zero.

\end{proof}

Let $H=(h_1,\ldots ,h_n):\R^n\longrightarrow \R^n$ be continuous, 
let $U\subset\R^n$ be an open set such that 
$H\inv (0)\cap  U$ is compact, so that the  topological degree 
$\deg(H,U,0)$ is defined. Suppose that $h_1(x)=x_1g(x)$, where 
$g(x)>0$ for $x\in U$. By $U'$ denote  the set 
$\{x'=(x_2,\ldots ,x_n)\in\R^{n-1}\ |\ (0,y')\in U\}$. Of course 
$H\inv (0)\cap U\subset \{0\}\times U'$. 
Let us define the mapping $H':\R^{n-1}\longrightarrow\R^{n-1}$  by 
$H'(x')=(h_2(0,x'),\ldots ,h_n(0,x'))$. 
Then $(H')^{-1} (0)\cap  U'$ is compact and $\deg(H',U',0)$ is well defined.
We have

\begin{lemma}\label{stopien}
$\deg(H,U,0)=\deg(H',U',0)$.
\end{lemma} \hspace*{\fill}$\Box$\par\medskip

Let us assume that
$f :\R ^{m+1} \longrightarrow \R $
is a smooth function such that
$M=f^{-1}(0)$ is compact  
and  $\nabla f(p)\neq 0$ at each $p\in M$, 
so that $M$ is an $m$-dimensional manifold. We shall say that vectors
$v_1, \ldots ,v_m\in T_pM$ are well oriented if vectors $\nabla
f(p),v_1,\ldots ,v_m$ are well oriented in
$\R^{m+1}$. In this way $M$ is an oriented manifold.

Let
$G=(g_1,\ldots,g_{2m}) :\R^ {m+1} \longrightarrow \R
^{2m}$
be smooth. 
Put $g=G|M$ and
 define
$H:\R^{m+1}\times
\R^{m+1}\longrightarrow \R^{2+2m}$ by
$$H(x,y)=(f(x),f(y),g_1(x)-g_1(y),\ldots
,g_{2m}(x)-g_{2m}(y)).$$
According to \cite[Lemma 18, Proposition 20]{KarNowSzafr} we have

\begin{proposition}\label{nasze}
The mapping $g:M\longrightarrow \R^{2m}$ is an immersion if and only if
the mapping $\R^{m+1}\ni x\mapsto (f(x),g_1(x),\ldots, g_{2m}(x))$ has rank $m+1$
at each $p\in M$.

If that is the case then there exists a compact
$2(m+1)$-dimensional manifold with boundary $N\subset\R^{m+1}\times\R^{m+1}$
such that 
$$\{(x,y)\in\R^{m+1}\times\R^{m+1}\ |\ H(x,y)=0, x\neq y\}\subset N\setminus\partial N.$$

If $m$ is even, then for any such $N$ the intersection number $I(g)$ equals
$\deg(H,N,0)/2= \deg(H|\partial N)/2.$ 
\end{proposition}\hspace*{\fill}$\Box$\par\medskip

From now on we assume that $g=G|M$ is an immersion. Then  

$$\rank\left [
\begin{array}{ccc}
\frac{\partial f}{\partial x_1}(x)&\ldots &\frac{\partial f}{\partial x_{m+1}}(x)\\
\frac{\partial g_1}{\partial x_1}(x)&\ldots &\frac{\partial g_1}{\partial x_{m+1}}(x)\\
&\ldots&\\
\frac{\partial g_{2m}}{\partial x_1}(x)&\ldots &\frac{\partial g_{2m}}{\partial x_{m+1}}(x)
\end{array} \right ]=m+1,$$ 
for $x\in M$. Denote by $\alpha_1(x), \ldots ,\alpha_{m+1}(x)$ the columns of the matrix above. 
This way with the  immersion $g$ we can associate 
$\alpha=(\alpha_1,\ldots ,\alpha_{m+1}):M\longrightarrow \widetilde{V}_{m+1}(\R^{2m+1})$, and 
$\widetilde{\alpha}=\beta_1\alpha_1(x)+\ldots + \beta_{m+1}\alpha_{m+1}(x):\R^{m+1}\times M\longrightarrow \R^{2m+1}$
such that $\widetilde{\alpha}|S^m\times M$ goes into $\R^{2m+1}\setminus \{0\}$.
By Theorem \ref{parzystoscstopnia}, the degree $\deg(\widetilde{\alpha}|S^m\times M)$ is well defined
and even.

Let us define $\phi:\R^{m+2}\times \R^{m+1}\longrightarrow \R^{2m+2}$ by 
$$\phi(\beta,\beta_{m+2};x)=(x+\beta,x+\beta_{m+2}\nabla f(x)),$$ 
where $\beta=(\beta_1,\ldots ,\beta_{m+1})$, $x=(x_1,\ldots ,x_{m+1})$. 

\begin{lemma}
For any $r>0$ small enough,  $\phi:B^{m+2}(r)\times M\rightarrow\R^{2m+2}$ is an orientation preserving diffeomorphism
onto its image.

\end{lemma}

\begin{proof}  Take a well oriented basis $v_1,\ldots , v_m$ of 
$T_xM$, so that $\nabla f(x),v_1,\ldots ,v_m$ are well oriented in $\R^{m+1}$. 
Let $e_1,\ldots ,e_{m+2}$ be the standard basis of $\R^{m+2}$. Take 
$q=(\beta,\beta_{m+2};x)\in B^{m+2}(r)\times M$.
Then 
$$(e_1,0),\ldots ,(e_{m+2},0),(0,v_1),\ldots ,(0,v_m)$$ 
is a  well oriented basis in $T_q (B^{m+2}(r)\times M)$. 
If $\beta_{m+2}=0$ then 
$$\det\big(\left[ D\phi((\beta,0;x))\right]\left[(e_1,0),\ldots ,(e_{m+2},0),(0,v_1),\ldots ,(0,v_m)\right]\Big)$$
$$=\det\left[\begin{array}{ccccccc}
1&\ldots&0&0&1&\ldots &0\\
&\ddots&&0&&\ddots&\\
0&\ldots&1&0&0&\ldots&1\\
0&\ldots&0&\frac{\partial f}{\partial x_1}(x)&1&\ldots &0\\
&\ddots&&\vdots&&\ddots&\\
0&\ldots&0&\frac{\partial f}{\partial x_{m+1}}(x)&0&\ldots&1\\
\end{array}\right]\left[(e_1,0),\ldots ,(e_{m+2},0),(0,v_1),\ldots ,(0,v_m)\right]$$
$$=\det\left[\begin{array}{ccccccc}
1&\ldots&0&0&&&\\
&\ddots&&0&&*&\\
0&\ldots&1&0&&&\\
0&\ldots&0&\nabla f(x)&v_1&\ldots &v_m\\

\end{array}\right]=\det\left[\nabla f(x),v_1,\ldots ,v_m\right]>0.$$ 
Since $\phi(0;x)=(x,x)$ and $M$ is compact,
if $r>0$ is small enough then $\phi:B^{m+2}(r)\times M\rightarrow\R^{2m+2}$ is an orientation preserving
diffeomorphism onto its image.
\end{proof}

%%%%%%%%%%%%%%%%%%%%%%%%%%%%%%%%%%%%%%%%%%%%%%%%%%%%%

\begin{lemma}\label{suma}
There exist smooth functions $u_1,\ldots ,u_{m+1}:\R^{m+1}\times \R^{m+1}\longrightarrow\R$ 
such that $f(x+y)=f(x)+\sum_1^{m+1}y_i u_i(x,y)$, and $u_i(x,0)=\frac{\partial f}{\partial x_i}(x)$.
\end{lemma}\hspace*{\fill}$\Box$\par\medskip

\begin{theorem}\label{immersje}
If $m$ is even and $g:M\rightarrow\R^{2m}$ is an immersion then
$I(g)=-\deg (\widetilde{\alpha}|S^m\times M)/2=-\Lambda(\alpha)$.
\end{theorem}

\begin{proof} 
Let $\Delta=\{(x,x)\ |\ x\in M\}$ denote the diagonal in $M\times M$.
Note that
$$H^{-1}(0)=\Delta\cup\{(x,y)\in M\times M,\ g(x)=g(y),\ x\neq y\}.$$
By Proposition \ref{nasze}, there is
$\varepsilon>0$ such that $|x-y|>\varepsilon$ for $(x,y)\in H^{-1}(0)\setminus
\Delta$.

Take $r>0$ such that $\phi:B^{m+2}(2r)\times M\rightarrow\R^{2m+2}$ is an orientation preserving diffeomorphism
onto its image.
Put $K=\phi(\bar{B}^{m+2}(r)\times M)$. Then $K$ is a closed tubular neighbourhood of $\phi(\{0\}\times M)=\Delta$ in $\R^{2m+2}$, 
and so $K$ is a $(2m+2)$--dimensional compact  manifold with boundary 
$\partial K=\phi(S^{m+1}(r)\times M)$. Moreover we can assume that 
$$|(x+\beta)-(x+\beta_{m+2}\nabla f(x))|<\varepsilon,\  \mbox{ for }(\beta,\beta_{m+2};x) \in \bar{B}^{m+2}(r)\times M.$$ 
In particular $\partial K\cap H\inv (0)=\emptyset$ and $K\cap H^{-1}(0)=\Delta$.  
For $R>0$ big enough, $N=\bar{B}^{2m+2}(R)\setminus \phi(B^{m+2}(r)\times M)$ is a compact manifold 
with boundary  $S^{2m+1}(R)\cup \partial K$, where the orientation of $\partial K$ is 
opposite to the one induced from $K$.  We may also assume that
$N$ contains
$H\inv(0)\setminus \Delta$ in its interior. According to Proposition \ref{nasze}, 
$$2I(g)=\deg(H,N,0)=\deg (H|S^{2m+1}(R))-\deg (H|\partial K).$$  
The hypersurface  $M=f\inv(0)$ is compact, so $\{f(x)=f(y)=0\}=M\times M\subset\R^{m+1}\times\R^{m+1}$ 
is compact too. According to Lemma \ref{stopienNaSferze}, $\deg (H|S^{2m+1}(R))=0$. So 
$$2I(g)=-\deg (H|\partial K)=-\deg (H\circ \phi|S^{m+1}(r)\times M)=-\deg(H\circ \phi,  B^{m+2}(r)\times M,0).$$
Of course it holds true for any radius smaller than $r$.

According to Lemma \ref{suma}, for $(\beta, \beta_{m+2};x)\in B^{m+2}(r)\times M$ 
the second coordinate of $H\circ\phi$ equals
$$f(x+\beta_{m+2}\nabla f(x))=f(x)+\beta_{m+2}\sum_{i=1}^{m+1}\frac{\partial f}{\partial x_i}(x)u_i(\beta, \beta_{m+2};x)$$
$$=\beta_{m+2}\sum_{i=1}^{m+1}\frac{\partial f}{\partial x_i}(x)u_i(\beta, \beta_{m+2};x),$$ 
where $u_i(0;x)=\frac{\partial f}{\partial x_i}(x)$. For $r$ small enough 
$\sum_{i=1}^{m+1}\frac{\partial f}{\partial x_i}(x)u_i(\beta, \beta_{m+2};x)>0$.
After permuting coordinates, by Lemma \ref{stopien}, we get
$$\deg(H\circ \phi,  B^{m+2}(r)\times M,0)$$
$$=\deg\left((f(x+\beta),g_1(x+\beta)-g_1(x),\ldots,g_{2m}(x+\beta)-g_{2m}(x)),B^{m+1}(r)\times M,0\right).$$ 

By Lemma \ref{suma}, for $(\beta,x)\in B^{m+1}(r)\times M$ we have 
$$\left(f(x+\beta),g_1(x+\beta)-g_1(x),\ldots,g_{2m}(x+\beta)-g_{2m}(x)\right)$$
$$=\left(\sum_1^{m+1}\beta_i u_i(\beta,x),\sum_1^{m+1} \beta_i  w_i(\beta,x)\right)=\sum_1^{m+1}\beta_i\left(u_i(\beta,x),w_i(\beta,x)\right),$$ 
where  $w_i(0,x)=(w_i^1(0,x),\ldots ,w_i^{2m}(0,x))=$ 
$(\frac{\partial g_1}{\partial x_i}(x),\ldots ,\frac{\partial g_{2m}}{\partial x_i}(x)).$ 
Because $g$ is an immersion, by Proposition \ref{nasze} there is small $r$  such that  

$$\rank\left [
\begin{array}{ccc}
u_1(\beta,x)&\ldots &u_{m+1}(\beta,x)\\
w_1^1(\beta,x)&\ldots &w_{m+1}^1(\beta,x)\\
 &\ldots & \\
w_1^{2m}(\beta,x)&\ldots &w_{m+1}^{2m}(\beta,x)\\
\end{array} \right ]=m+1,$$ 
for $(\beta, x)\in \bar{B}^{m+1}(r)\times M$.
Hence the columns $\alpha_1(\beta,x),\ldots,\alpha_{m+1}(\beta,x)$ of the matrix above are linearly independent.
Let
$h_t:\bar{B}^{m+1}(r)\times M\longrightarrow\R^{2m+1}$, $0\leq t\leq 1$, be a homotopy given by 
$$h_t(\beta,x)=\beta_1\alpha_1(t\beta,x)+\ldots +\beta_{m+1}\alpha_{m+1}(t\beta,x),$$ 
so that $h_t(\beta,x)$ is a linear combination of linearly independent vectors.
Then each $h_t^{-1}(0)=\{0\}\times M$.

According to the Excision Theorem  we have  
$$\deg (\widetilde{\alpha}|S^m\times M)=\deg(\widetilde{\alpha},\bar{B}^{m+1}(1)\times M,0)=\deg(\widetilde{\alpha},\bar{B}^{m+1}(r)\times M,0).$$ 
Of course $\widetilde{\alpha}=h_0$, so $\deg(\widetilde{\alpha}|S^m\times M)=\deg(h_0,\bar{B}^{m+1}(r)\times M,0)=
\deg(h_1,\bar{B}^{m+1}(r)\times M,0)$. 
By the previous arguments,  $\deg (h_1,\bar{B}^{m+1}(r)\times M,0)=-2I(g)$. 
To sum up we get that $2I(g)=-\deg (\widetilde{\alpha}|S^m\times M)=-2\Lambda(\alpha)$.

\end{proof}

\begin{example} Let $g=(x_3^3+x_2-x_1-3x_3,x_2^3+2x_1-x_2+x_3,x_1x_2+2x_1,x_1x_3-x_2):\R^3\longrightarrow\R^4$. 
Using {\sc Singular}  \cite{GPS06} and results of Theorems \ref{efeektywnie} and \ref{immersje} one may check that $I(g|S^2(10))=5$.
\end{example}
%\begin{acknowledgements}
%\end{acknowledgements}


\begin{thebibliography}{99}
\bibitem{becker} Becker E., W\"{o}rmann T.: On the trace formula for quadratic forms and some applications,
 Contemporary Mathematics \textbf{155}, 271-291, (1994)
\bibitem{golub} Golubitsky M., Guillemin V.: Stable mappings and their singularities,  Springer-Verlag New York, (1973)
\bibitem{GPS06}
Greuel G.-M., Pfister G., Sch\"onemann H.:
\newblock {{\sc Singular} 3.0.2}. A Computer Algebra System for
Polynomial Computations
\bibitem{hatcher} Hatcher A.: Vector Bundles and K-theory,\\\url{http://www.math.cornell.edu/~hatcher/VBKT/VBpage.html}
\bibitem{KarNowSzafr} Karolkiewicz I., Nowel A., Szafraniec Z.: An algebraic formula for the intersection
number of a polynomial immersion, J. Pure Appl. Algebra \textbf{214} (3), 269-280, ( 2010)
\bibitem{krzyzanowska} Krzy\.zanowska I.: The intersection number of real polynomial mappings,
Topology and its Applications \textbf{158}, 379-386, (2011)
%\bibitem{LashofSmale} Lashof R. and Smale S., On the immersion of
%manifolds in Euclidean spaces, Annals of Mathematics \textbf{68} (3),
%562-583, (1958)
 \bibitem{pedersenetal} Pedersen P., Roy M.-F., Szpirglas A.: Counting real zeros in the multivariate case, Computational Algebraic Geometry, Progr. in Math. \textbf{109},  203-224, Birkh\"{a}user, (1993)

\bibitem{Smale} Smale S.: The classification of immersions of
spheres in Euclidean spaces, Annals of mathematics \textbf{69} (2), 327-344, (1959)
\bibitem{Szafraniec} Szafraniec Z.: Topological degree and quadratic
forms, Journal of Pure and Applied Algebra \textbf{141}, 299-314, (1999)
\bibitem{whitneySelfInter} Whitney H.: The self-intersections of a
smooth $n$-manifold in $2n$-space, Annals of Mathematics \textbf{45} (2), 220-246 (1944)
\end{thebibliography}


\end{document}